\documentclass{amsart}
\usepackage{amsfonts, amsbsy, amsmath, amssymb}
\newtheorem{thm}{Theorem}[section]
\newtheorem{lem}[thm]{Lemma}
\newtheorem{cor}[thm]{Corollary}
\newtheorem{prop}[thm]{Proposition}

\numberwithin{equation}{section}

\theoremstyle{definition}

\begin{document}

\title[Third Power of the Reversed Dickson Polynomial]{Third Power of the Reversed Dickson Polynomial over Finite Fields}

\author{Xiang-dong Hou}
\address{Department of Mathematics and Statistics,
University of South Florida, Tampa, FL 33620}
\email{xhou@cas.usf.edu}


 
\begin{abstract}
Let $D_n(1,x)$ be the $n$th reversed Dickson polynomial. The power sums $\sum_{a\in\Bbb F_q}D_n(1,a)^i$, $i=1,2$, have been determined recently. In this paper we give an evaluation of the sum $\sum_{a\in\Bbb F_q}D_n(1,a)^3$. This result implies new necessary conditions for $D_n(1,x)$ to be a permutation polynomial over $\Bbb F_q$.
\end{abstract}

\maketitle


\section{Introduction}

Let $n\ge 0$ be an integer and let $\Bbb F_q$ denote the finite field with $q$ elements. The $n$th {\em reversed Dickson polynomial} $D_n(1,x)\in\Bbb Z[x]$ is defined by the functional equation
\[
D_n(1,x(1-x))=x^n+(1-x)^n.
\]
Naturally, $D_n(1,x)$ can be viewed as a polynomial over $\Bbb F_q$. The reversed Dickson polynomial is a descendant of the polynomial $D_n(x,y)\in\Bbb Z[x,y]$ defined by the functional equation
\[
D_n(x+y,xy)=x^n+y^n.
\]
The other descendant of $D_n(x,y)$ is the well known {\em Dickson polynomial} $D_n(x,a)\in \Bbb F_q[x]$ where $a\in\Bbb F_q$. While Dickson polynomials have been the focus of many researchers for over a century (cf. \cite{LMT}), the significance of reversed Dickson polynomials over finite fields was not clear until some ten years ago. In \cite{Dillon02}, Dillon explored a connection between reversed Dickson polynomials that are permutations of $\Bbb F_{2^m}$ and {\em almost perfect nonlinear} (APN) functions over $\Bbb F_{2^m}$. A more comprehensive approach to reversed Dickson polynomials as permutation polynomials over finite fields appeared in a recent paper \cite{HMSY}. We are interested in the pairs $(q,n)$ for which $D_n(1,x)$ is a permutation polynomial over $\Bbb F_q$ and we call such pairs {\em desirable} \cite{Hou10}. As explained in the introduction of \cite{HL}, when searching for desirable pairs $(q,n)$, we may assume $1\le n\le q^2-2$. All known families (ten families) of desirable pairs are listed in Table 1 of \cite{HL}. Computer search has confirmed that there are no other desirable pairs for $q\le 401$. So the big open question is whether the list of known desirable pairs is complete. Any new addition to the list would be extremely interesting since most families in the list are already highly nontrivial. Another way to attack the problem is to find new necessary conditions for a pair to be desirable. It is well known that a function $f:\Bbb F_q\to \Bbb F_q$ is bijective if and only if 
\[
\sum_{a\in\Bbb F_q}f(a)^i
\begin{cases}
=0&\text{if}\ 1\le i\le q-2,\cr
\ne 0&\text{if}\ i=q-1.
\end{cases}
\]
Therefore, an explicit evaluation of the sum $\sum_{a\in\Bbb F_q}D_n(1,a)^i$ for any $1\le i\le q-1$ would provide necessary conditions for $(q,n)$ to be desirable. The sums
$\sum_{a\in\Bbb F_q}D_n(1,a)^i$, $i=1,2$, have been determined in \cite{HL} for this purpose. In the present paper, we give an explicit evaluation of the sum $\sum_{a\in\Bbb F_q}D_n(1,a)^3$. As expected, this result implies new necessary conditions for $(q,n)$ to be desirable.

In Section 2 we recall certain results from \cite{HL} to be used in the present paper. The evaluation of the sum $\sum_{a\in\Bbb F_q}D_n(1,a)^3$ requires 
separate treatments of
the even $q$ case and the odd $q$ case. These two cases are covered in Sections 3 and 4 respectively.


\section{Preliminaries}

We first prove a lemma

\begin{lem}\label{L2.1}
Let $q=p^e$, where $p$ is a prime and $e>0$. Let $0\le \alpha, v\le q-1$. Then in $\Bbb F_p$,
\[
\binom{q-1+\alpha-v}\alpha=(-1)^\alpha\binom v\alpha.
\]
\end{lem}

\begin{proof}
Write $\alpha=\sum_{i=0}^{e-1}\alpha_ip^i$, $v=\sum_{i=0}^{e-1}v_ip^i$, $0\le \alpha_i, v_i\le p-1$. If $v_i<\alpha_i$ for some $i$, then $\binom v\alpha=0$ and $\binom{q-1+\alpha-v}\alpha=0$. (The second equation holds since the sum $\alpha+(q-1-v)$ has carry in base $p$.) Now assume $v_i\ge \alpha_i$ for all $0\le i\le e-1$. Then
\[
\binom{q-1+\alpha-v}\alpha =\prod_{i=0}^{e-1}\binom{p-1+\alpha_i-v_i}{\alpha_i}=\prod_{i=0}^{e-1}(-1)^{\alpha_i}\binom{v_i}{\alpha_i}
=(-1)^\alpha\binom v\alpha.
\]
\end{proof}

Let $d_n=D_n(1,x)$. Since
\[
\begin{split}
d_n(x(1-x))^3\;&=\bigl[x^n+(1-x)^n\bigr]^3\cr
&=x^{3n}+(1-x)^{3n}+3\bigl[x(1-x)\bigr]^n\bigl[x^n+(1-x)^n\bigr]\cr
&=d_{3n}(x(1-x))+3\bigl[x(1-x)\bigr]^nd_n(x(1-x)),
\end{split}
\]
we have
\begin{equation}\label{2.1}
d_n^3=d_{3n}+3x^nd_n.
\end{equation}
Therefore
\begin{equation}\label{2.2}
\sum_{a\in\Bbb F_q}d_n(a)^3=\sum_{a\in\Bbb F_q}d_{3n}(a)+3\sum_{a\in\Bbb F_q}a^nd_n(a).
\end{equation}
In \eqref{2.2}, the sum $\sum_{a\in\Bbb F_q}d_{3n}(a)$ has been determined in \cite{HL}; the goal of the present paper is to evaluate the sum $\sum_{a\in\Bbb F_q}a^nd_n(a)$.

By (4.3) of \cite{HL},
\begin{equation}\label{2.3}
\sum_{n=1}^{q^2-1}d_nt^n\equiv\frac{t(t^{q^2-1}-1)}{t-1}+h(t)\sum_{k=1}^{q-1}(t-1)^{q-1-k}t^{2k}x^k\pmod{x^q-x},
\end{equation}
where
\begin{equation}\label{2.4}
h(t)=\frac{(t-2)(t^{q^2-1}-1)}{(t^{q-1}-1)(t^q-t^{q-1}-1)}\in\Bbb F_p[t]\qquad (p=\text{char}\;\Bbb F_q).
\end{equation}
Summing both sides of \eqref{2.3} as $x$ runs over $\Bbb F_q$, we get
\begin{equation}\label{2.5}
\sum_{n=2(q-1)}^{q^2-1}\Bigl(\sum_{a\in\Bbb F_q}d_n(a)\Bigr)t^n=-t^{2(q-1)}h(t).
\end{equation}

\begin{lem}[{\cite [Proposition 4.1]{HL}}]\label{L2.2}
Assume that $q$ is even. Then
\begin{equation}\label{2.6}
h(t)=\sum_{\substack{\alpha,\beta\ge 0\cr \alpha+\beta\le q-2}}\binom{\alpha+\beta}\alpha t^{(q-1)^2-(\alpha+\beta q)}.
\end{equation}
\end{lem}

\begin{lem}\label{L2.3}
Assume that $q$ is odd. Then
\begin{equation}\label{2.7}
h(t)=t^{(q-1)^2}+\sum_{\substack{\alpha,\beta\ge 0\cr \beta\le q-2\cr 0<\alpha+\beta\le q-1}}\Bigl[2^{\alpha+\beta}\binom{2(q-1)-\alpha-\beta}{q-1}-\binom{\alpha+\beta}\alpha\Bigr]t^{(q-1)^2-(\alpha+\beta q)}.
\end{equation}
\end{lem}

\begin{proof}
Let $f_n(x)=\sum_{j\ge 0}\binom n{2j}x^j$. Let $n=\alpha+\beta q$, where $0\le \alpha,\beta\le q-1$. When $1\le n\le q^2-2$, we have
\[
\begin{split}
&\sum_{a\in\Bbb F_q}d_n(a)\cr
=\;&\Bigl(\frac 12\Bigr)^{n-1}\sum_{a\in\Bbb F_q}f_n(a)\kern 5.8cm\text{(by \cite[Proposition 2.1]{HL})}\cr
=\;&\Bigl(\frac 12\Bigr)^{n-1}\Bigl(-\frac 12\Bigr)\Bigl[\binom{\alpha+\beta}{q-1}-2^{\alpha+\beta}\binom{2(q-1)-\alpha-\beta}{q-1-\alpha}\Bigr]\kern 5mm
\text{(by \cite[Theorem 3.1]{HL})}\cr
=\;&-2^{-\alpha-\beta}\binom{\alpha+\beta}{q-1}+\binom{2(q-1)-\alpha-\beta}{q-1-\alpha}.
\end{split}
\]
When $n=q^2-1$, since $d_{q^2-1}(x)\equiv x^{q-1}+1\pmod{x^q-x}$, we have 
\[
\sum_{a\in\Bbb F_q}d_{q^2-1}(a)=-1.
\]
To sum up, we have
\begin{equation}\label{2.8}
\sum_{a\in\Bbb F_q}d_n(a)=
\begin{cases}
\displaystyle -2^{-\alpha-\beta}\binom{\alpha+\beta}{q-1}+\binom{2(q-1)-\alpha-\beta}{q-1-\alpha}&\text{if}\ 1\le n\le q^2-2, \vspace{3mm}\cr
-1&\text{if}\ n=q^2-1.
\end{cases}
\end{equation}
Therefore,
\[
\begin{split}
&h(t)\cr
=\;&-\sum_{n=2(q-1)}^{q^2-1}\Bigl(\sum_{a\in\Bbb F_q}d_n(a)\Bigr)t^{n-2(q-1)}\kern 2cm\text{(by \eqref{2.5})}\cr
=\;&-\sum_{i=0}^{(q-1)^2}\Bigl(\sum_{a\in\Bbb F_q}d_{q^2-1-i}(a)\Bigr)t^{(q-1)^2-i}\kern 1.7cm (n=q^2-1-i)\cr
=\;&t^{(q-1)^2}+\sum_{\substack{0\le \alpha,\beta\le q-1\cr 0<\alpha+\beta q\le(q-1)^2}}\Bigl[2^{\alpha+\beta}\binom{2(q-1)-\alpha-\beta}{q-1}-\binom{\alpha+\beta}\alpha\Bigr]t^{(q-1)^2-(\alpha+\beta q)}\cr
&\kern 7cm \text{($i=\alpha+\beta q$ and by \eqref{2.8})}.
\end{split}
\]
In the above sum, if $\alpha+\beta\ge q$, then $\binom{2(q-1)-\alpha-\beta}{q-1}=0$ and $\binom{\alpha+\beta}\alpha=0$. (The second equation holds since the sum $\alpha+\beta$ has carry in base $p$.) Therefore \eqref{2.7} follows.
\end{proof}

In \eqref{2.3} substitute $t$ by $xt$. We then have
\[
\begin{split}
&\sum_{n=1}^{q^2-1}d_nx^nt^n\cr
\equiv\;&\frac{xt[(xt)^{q^2-1}-1]}{xt-1}+h(xt)\sum_{k=1}^{q-1}(xt-1)^{q-1-k}(xt)^{2k}x^k\pmod{x^q-x}\cr
\equiv\;&\frac{xt(1-t^{q^2-1})}{1-xt}+h(xt)\sum_{k=1}^{q-1}(xt-1)^{q-1-k}t^{2k}x^{3k}\pmod{x^q-x}. 
\end{split}
\]
Since 
\[
\begin{split}
\frac 1{1-xt}\;&=\sum_{k\ge 0}x^kt^k\cr
&=1+\sum_{k=1}^{q-1}\sum_{l\ge 0}x^{k+l(q-1)}t^{k+l(q-1)}\cr
&\equiv 1+\sum_{k=1}^{q-1}x^kt^k\sum_{l\ge 0}t^{l(q-1)}\pmod{x^q-x}\cr
&=1+\frac 1{1-t^{q-1}}\sum_{k=1}^{q-1}x^kt^k,
\end{split}
\]
we have
\begin{equation}\label{2.9}
\begin{split}
&\sum_{n=1}^{q^2-1}d_nx^nt^n\cr
\equiv\;&xt(1-t^{q^2-1})\Bigl[1+\frac 1{1-t^{q-1}}\sum_{k=1}^{q-1}x^kt^k\Bigr]+h(xt)\sum_{k=1}^{q-1}(xt-1)^{q-1-k}t^{2k}x^{3k}\pmod{x^q-x}\cr
=\;&(t-t^{q^2})x+\frac{t^{q^2-1}-1}{t^{q-1}-1}\sum_{k=1}^{q-1}t^{k+1}x^{k+1}+h(xt)\sum_{k=1}^{q-1}(xt-1)^{q-1-k}t^{2k}x^{3k}.
\end{split}
\end{equation}
From here on, the even $q$ case and the odd $q$ case have to be considered separately.

\section{The Even $q$ Case}

Assume that $q$ is even. By Lemma~\ref{L2.2},
\begin{equation}\label{3.1}
\begin{split}
h(xt)\;&=\sum_{\substack{\alpha,\beta\ge 0\cr \alpha+\beta\le q-2}}\binom{\alpha+\beta}\alpha t^{(q-1)^2-(\alpha+\beta q)}x^{(q-1)^2-(\alpha+\beta q)}\cr
&\equiv\sum_{\substack{\alpha,\beta\ge 0\cr \alpha+\beta\le q-2}}\binom{\alpha+\beta}\alpha t^{(q-1)^2-(\alpha+\beta q)}x^{q-1-(\alpha+\beta)}\pmod{x^q-x}.
\end{split}
\end{equation}
By \eqref{2.9} and \eqref{3.1},
\[
\begin{split}
&\sum_{n=1}^{q^2-1}d_nx^nt^n\cr
\equiv\;&(t-t^{q^2})x+\frac{t^{q^2-1}-1}{t^{q-1}-1}\sum_{k=1}^{q-1}t^{k+1}x^{k+1}\cr
&+\sum_{\substack{\alpha,\beta\ge 0\cr \alpha+\beta\le q-2}}\sum_{k=1}^{q-1}\binom{\alpha+\beta}\alpha(xt-1)^{q-1-k}t^{(q-1)^2+2k-(\alpha+\beta q)}x^{q-1+3k-(\alpha+\beta)}\cr
&\kern 8cm\pmod{x^q-x}\cr
=\;&(t-t^{q^2})x+\frac{t^{q^2-1}-1}{t^{q-1}-1}\sum_{k=1}^{q-1}t^{k+1}x^{k+1}\cr
&+\sum_{\substack{\alpha,\beta\ge 0\cr \alpha+\beta\le q-2}}\sum_{k=1}^{q-1}\sum_j\binom{\alpha+\beta}\alpha\binom{q-1-k}j(xt)^{q-1-k-j}t^{(q-1)^2+2k-(\alpha+\beta q)}x^{q-1+3k-(\alpha+\beta)}\cr
=\;&(t-t^{q^2})x+\frac{t^{q^2-1}-1}{t^{q-1}-1}\sum_{k=1}^{q-1}t^{k+1}x^{k+1}\cr
&+\sum_{\substack{\alpha,\beta\ge 0\cr \alpha+\beta\le q-2}}\sum_{k=1}^{q-1}\sum_j\binom{\alpha+\beta}\alpha\binom{q-1-k}j 
t^{q(q-1)+k-j-(\alpha+\beta q)}x^{2(q-1)+2k-j-(\alpha+\beta)}.
\end{split}
\]
Let $x$ vary over $\Bbb F_q$ and sum both sides of the above equation. We get
\[
\begin{split}
&\sum_{n=1}^{q^2-1}\Bigl(\sum_{a\in\Bbb F_q}a^nd_n(a)\Bigr)t^n\cr
=\;&\frac{t^{q^2-1}-1}{t^{q-1}-1}t^{q-1}+\sum_{\substack{\alpha,\beta,j\ge 0; k\ge 1\cr \alpha+\beta\le q-2; j+k\le q-1\cr
2k-j-(\alpha+\beta)\equiv 0 \kern-2mm \pmod{q-1}}}\binom{\alpha+\beta}\alpha\binom{q-1-k}j t^{q(q-1)+k-j-(\alpha+\beta q)}\cr
=\;&\sum_{l=1}^{q+1}t^{l(q-1)}+\sum_{\substack{\alpha,\beta,j\ge 0; k\ge 1\cr \alpha+\beta\le q-2; j+k\le q-1\cr
2k-j-(\alpha+\beta)\equiv 0 \kern-2mm \pmod{q-1}}}\binom{\alpha+\beta}\alpha\binom{q-1-k}j t^{q(q-1)+k-j-(\alpha+\beta q)}.
\end{split}
\]
Therefore, for $1\le n\le q^2-1$,
\begin{equation}\label{3.2}
\sum_{a\in\Bbb F_q}a^nd_n(a)=\delta_n+\sum_{\substack{\alpha,\beta,j\ge 0;k\ge 1\cr \alpha+\beta\le q-2;j+k\le q-1\cr 2k-j-(\alpha+\beta)\equiv 0 \kern-2mm \pmod{q-1}\cr q(q-1)+k-j-(\alpha+\beta q)=n}}\binom{\alpha+\beta}\alpha\binom{q-1-k}j,
\end{equation}
where
\begin{equation}\label{3.2a}
\delta_n=
\begin{cases}
1&\text{if}\ n\equiv 0\pmod {q-1},\cr
0&\text{if}\ n\not\equiv 0\pmod {q-1}.
\end{cases}
\end{equation}

\begin{prop}\label{P3.1}
Let $q$ be a power of $2$. Let $1\le n\le q^2-1$ and write $n=u+vq$, $0\le u,v\le q-1$. Then in $\Bbb F_2$,
\begin{equation}\label{3.4}
\begin{split}
&\sum_{a\in\Bbb F_q}a^nd_n(a)\cr
=\;&\delta_n+\sum_{-1\le s\le 2}\;\sum_{\max\{s-\frac{u+v}{q-1},v-q+1\}\le \epsilon<\min\{s-\frac{u+v}{q-1}+1,v\}}
\binom{u+v-(s-\epsilon)(q-1)}{(s-2\epsilon+1)(q-1)-2u-v-\epsilon}.
\end{split}
\end{equation}
\end{prop}

\begin{proof}
First note that in the sum on the right side of \eqref{3.2}, the condition $q(q-1)+k-j-(\alpha+\beta q)=n$ is equivalent to $k-j-\alpha-u+(q-1-\beta-v)q$, which is further equivalent to
\[
\begin{cases}
k-j-\alpha-u=\epsilon q,\cr
q-1-\beta-v+\epsilon=0
\end{cases}
\] 
for some $\epsilon\in \Bbb Z$. Therefore, the conditions on $\alpha,\beta,j,k$ in that sum can be {\em replaced} by
\begin{equation}\label{3.7}
\begin{cases}
\beta\ge 0,\cr
1\le k\le q-1,\cr
0\le \alpha\le q-2-\beta,\cr
2k-j-(\alpha+\beta)=s(q-1),\quad -1\le s\le 2,\cr
k-j-\alpha-u=\epsilon q,\quad \epsilon\in\Bbb Z,\cr
q-1-\beta-v+\epsilon=0.
\end{cases}
\end{equation}
(Note. We remind the reader that \eqref{3.7} is not equivalent to but weaker than the conditions in the sum in \eqref{3.2}; 
the restriction on $j$ is not present in \eqref{3.7}.
However, the relaxation only brings additional zero terms to the sum in \eqref{3.2}.) We solve \eqref{3.7} for $\beta,k,j$ in terms of $\alpha,s,\epsilon$. The result is
\begin{equation}\label{3.8}
\begin{cases}
\beta=q-1+\epsilon-v,\cr
k=(s-\epsilon+1)(q-1)-u-v,\cr
j=(s-2\epsilon+1)(q-1)-2u-v-\alpha-\epsilon,\cr
-1\le s\le 2, \vspace{1mm}\cr
\displaystyle \frac{u+v}{q-1}-1<s-\epsilon\le\frac{u+v}{q-1},\vspace{1mm}\cr
\epsilon\ge v-q+1,\cr
0\le \alpha\le v-1-\epsilon.
\end{cases}
\end{equation} 
Now \eqref{3.2} can be written as
\begin{equation}\label{3.9}
\begin{split}
&\delta_n+\sum_{a\in\Bbb F-q}a^nd_n(a)\cr
=\;&\sum_{\substack{s,\epsilon\cr -1\le s\le 2\cr \frac{u+v}{q-1}-1<s-\epsilon\le\frac{u+v}{q-1},\cr \epsilon\ge v-q+1}}\sum_{0\le \alpha\le v-1-\epsilon}
\binom{\alpha+q-1+\epsilon-v}\alpha\binom{u+v-(s-\epsilon)(q-1)}{(s-2\epsilon+1)(q-1)-2u-v-\alpha-\epsilon}\cr
=\;&\sum_{\substack{s,\epsilon\cr -1\le s\le 2\cr \frac{u+v}{q-1}-1<s-\epsilon\le\frac{u+v}{q-1},\cr v-q+1\le\epsilon\le v-1}}\sum_{0\le \alpha\le v-1-\epsilon}
\binom{q-1+\alpha-(v-\epsilon)}\alpha\binom{u+v-(s-\epsilon)(q-1)}{(s-2\epsilon+1)(q-1)-2u-v-\epsilon-\alpha}\cr
=\;&\sum_{\substack{s,\epsilon\cr -1\le s\le 2\cr \frac{u+v}{q-1}-1<s-\epsilon\le\frac{u+v}{q-1},\cr v-q+1\le\epsilon\le v-1}}\sum_{0\le \alpha\le v-1-\epsilon}
\binom{v-\epsilon}\alpha\binom{u+v-(s-\epsilon)(q-1)}{(s-2\epsilon+1)(q-1)-2u-v-\epsilon-\alpha}\cr
&\kern 9.5cm\text{(by Lemma~\ref{L2.1})} \cr
=\;&\sum_{\substack{s,\epsilon\cr -1\le s\le 2\cr \frac{u+v}{q-1}-1<s-\epsilon\le\frac{u+v}{q-1},\cr v-q+1\le\epsilon\le v-1}}\left[
\binom{u+v-(s-\epsilon)(q-1)}{(s-2\epsilon+1)(q-1)-2u-2v}\right.\cr
&\kern2.8cm +\sum_{0\le \alpha\le v-\epsilon} \left.
\binom{v-\epsilon}\alpha\binom{u+v-(s-\epsilon)(q-1)}{(s-2\epsilon+1)(q-1)-2u-v-\epsilon-\alpha}\right]\cr
=\;&\sum_{\substack{s,\epsilon\cr -1\le s\le 2\cr \frac{u+v}{q-1}-1<s-\epsilon\le\frac{u+v}{q-1},\cr v-q+1\le\epsilon\le v-1}}
\left[\binom{u+v-(s-\epsilon)(q-1)}{(s-2\epsilon+1)(q-1)-2u-2v}+\binom{u+v-(s-\epsilon)(q-1)+v-\epsilon}{(s-2\epsilon+1)(q-1)-2u-2v+v-\epsilon}\right]\cr
\end{split}
\end{equation}
\[
\begin{split}
=\;&\sum_{\substack{s,\epsilon\cr -1\le s\le 2\cr \frac{u+v}{q-1}-1<s-\epsilon\le\frac{u+v}{q-1},\cr v-q+1\le\epsilon\le v-1}}
\binom{u+v-(s-\epsilon)(q-1)}{(s-2\epsilon+1)(q-1)-2u-2v+v-\epsilon}\cr
=\;&\sum_{-1\le s\le 2}\;\sum_{\max\{s-\frac{u+v}{q-1},v-q+1\}\le \epsilon<\min\{s-\frac{u+v}{q-1}+1,v\}}
\binom{u+v-(s-\epsilon)(q-1)}{(s-2\epsilon+1)(q-1)-2u-v-\epsilon}.
\end{split}
\]
\end{proof}

\noindent{\bf Note.} In \eqref{3.4}, for each $-1\le s\le 2$, there is at most one $\epsilon$ in the specified range. Thus the sum in \eqref{3.4} contains at most four terms. 
More precisely, \eqref{3.4} can be stated as follows.

When $0<u+v<q-1$, 
\begin{equation}\label{3.10a}
\begin{split}
\sum_{a\in\Bbb F_q}a^nd_n(a)\;&=\sum_{-1\le s\le\min\{2,v-1\}}\binom{u+v}{(-s+1)(q-1)-2u-v-s}\cr
&=\sum_{-1\le s\le\min\{0,v-1\}}\binom{u+v}{(-s+1)(q-1)-2u-v-s}\cr
&=
\begin{cases}
\displaystyle \binom u{2(q-1)-2u+1}&\text{if}\ v=0, \vspace{3mm}\cr
\displaystyle \binom{u+v}{2(q-1)-2u-v+1}+\binom{u+v}{q-1-2u-v}&\text{if}\ v>0.
\end{cases}
\end{split}
\end{equation}

When $q-1\le u+v<2(q-1)$,
\begin{equation}\label{3.11a}
\begin{split}
\sum_{a\in\Bbb F_q}a^nd_n(a)\;&=\delta_n+\sum_{\max\{-1,v-q+2\}\le s\le\min\{2,v\}}\binom{u+v-(q-1)}{(-s+3)(q-1)-2u-v-s+1}\cr
&=\delta_n+\sum_{\max\{0,v-q+2\}\le s\le\min\{1,v\}}\binom{u+v-(q-1)}{(-s+3)(q-1)-2u-v-s+1}\cr
&=\begin{cases}
1&\text{if}\ v=0, \vspace{2mm}\cr
\displaystyle \delta_n\!+\!\binom{u+v-(q-1)}{3(q-1)\!-\!2u\!-\!v\!+\!1}\!+\!\binom{u+v-(q-1)}{2(q-1)-2u-v}&\text{if}\ 1\le v\le q-2, \vspace{2mm}\cr
\displaystyle \delta_n+\binom u{q-1-2u}&\text{if}\ v=q-1.
\end{cases}
\end{split}
\end{equation}

\begin{thm}\label{T3.2}
Let $q$ be a power of $2$ and let $1\le n\le q^2-1$ Write $n=u+vq$, $0\le u,v\le q-1$, and write $3n\equiv u'+v'q\pmod{q^2-1}$, $0\le u',v'\le q-1$. Then
\begin{equation}\label{}
\begin{split}
&\sum_{a\in\Bbb F_q}d_n(a)^3\cr
=\;&\delta_n+\sum_{-1\le s\le 2}\;\sum_{\max\{s-\frac{u+v}{q-1},v-q+1\}\le \epsilon<\min\{s-\frac{u+v}{q-1}+1,v\}}
\binom{u+v-(s-\epsilon)(q-1)}{(s-2\epsilon+1)(q-1)-2u-v-\epsilon}\cr 
&+
\begin{cases}
\displaystyle \binom{2(q-1)-u'-v'}{q-1-u'}&\text{if}\ u'+v'\ge q, \vspace{2mm}\cr
0&\text{if}\ u'+v'<q.
\end{cases}
\end{split}
\end{equation}
\end{thm}

\begin{proof}
By \eqref{2.2},
\[
\sum_{a\in\Bbb F_q}d_n(a)^3=\sum_{a\in\Bbb F_q}a^nd_n(a)+\sum_{a\in\Bbb F_q}d_{3n}(a).
\]
In the above, $\sum_{a\in\Bbb F_q}a^nd_n(a)$ is given by \eqref{3.4}. By \cite[Theorem 4.2]{HL},
\begin{equation}\label{3.11}
\sum_{a\in\Bbb F_q}d_{3n}(a)=\begin{cases}
\displaystyle \binom{2(q-1)-u'-v'}{q-1-u'}&\text{if}\ u'+v'\ge q, \vspace{2mm}\cr
0&\text{if}\ u'+v'<q.
\end{cases}
\end{equation}
\end{proof}

\begin{cor}\label{C3.3}
In Theorem~\ref{T3.2} assume $q>4$ and $(q,n)$ is desirable. Then $0<u+v<q-1$ or $q-1<u+v<2(q-1)$.

\begin{itemize}
\item[(i)]
If $0<u+v<q-1$, then in $\Bbb F_2$,
\begin{equation}\label{3.12}
\begin{split}
&\sum_{-1\le s\le\min\{0,v-1\}}\binom{u+v}{(-s+1)(q-1)-2u-v-s}\cr
=\;&\begin{cases}
\displaystyle \binom{2(q-1)-u'-v'}{q-1-u'}&\text{if}\ u'+v'\ge q,\vspace{2mm}\cr
0&\text{if}\ u'+v'<q.
\end{cases}
\end{split}
\end{equation}

\item[(ii)]
If $q-1<u+v<2(q-1)$, then in $\Bbb F_2$,
\begin{equation}\label{3.13}
\begin{split}
&\sum_{\max\{0,v-q+2\}\le s\le\min\{1,v\}}\binom{u+v-(q-1)}{(-s+3)(q-1)-2u-v-s+1}\cr
=\;&\begin{cases}
\displaystyle \binom{2(q-1)-u'-v'}{q-1-u'}&\text{if}\ u'+v'\ge q, \vspace{2mm}\cr
0&\text{if}\ u'+v'<q.
\end{cases}
\end{split}
\end{equation}
\end{itemize}
\end{cor}

\begin{proof}
By \cite[Theorem 2.5]{HL}, $(n,q^2-1)=3<q-1$. Thus $u+v\not\equiv 0\pmod{q-1}$, namely, $0<u+v<q-1$ or $q-1<u+v<2(q-1)$. Since $\sum_{a\in\Bbb F_q}a^nd_n(a)+\sum_{a\in\Bbb F_q}d_{3n}(a)=\sum_{a\in\Bbb F_q}d_n(a)^3=0$, \eqref{3.10a} and \eqref{3.11} yield \eqref{3.12}; \eqref{3.11a} and \eqref{3.11} yield \eqref{3.13}.
\end{proof}


\section{The Odd $q$ Case}

Assume that $q$ is an odd prime power. The plan of this section is parallel to that of Section 3. Starting from the end of Section 2, we proceed to determine the sum $\sum_{a\in\Bbb F_q}a^nd_n(a)$.

By Lemma 2.3,
\begin{equation}\label{4.1}
\begin{split}
&h(xt)\cr
=\;&(xt)^{(q-1)^2}\cr
&+\sum_{\substack{\alpha,\beta\ge 0; \beta\le q-2\cr 0<\alpha+\beta\le q-1}}\Bigl[2^{\alpha+\beta}\binom{2(q-1)-\alpha-\beta}{q-1}-\binom{\alpha+\beta}\alpha\Bigr]t^{(q-1)^2-(\alpha+\beta q)}x^{(q-1)^2-(\alpha+\beta q)}\cr
\equiv\;& t^{(q-1)^2}x^{q-1}\cr
&+\sum_{\substack{\alpha,\beta\ge 0; \beta\le q-2\cr 0<\alpha+\beta\le q-1}}\Bigl[2^{\alpha+\beta}\binom{2(q-1)-\alpha-\beta}{q-1}-\binom{\alpha+\beta}\alpha\Bigr]t^{(q-1)^2-(\alpha+\beta q)}x^{2(q-1)-(\alpha+\beta)}\cr
&\kern 9cm\pmod{x^q-x}.
\end{split}
\end{equation}
By \eqref{2.9} and \eqref{4.1},
\[
\begin{split}
&\sum_{n=1}^{q^2-1}d_nx^nt^n\cr
\equiv\;&(t-t^{q^2})x+\frac{t^{q^2-1}-1}{t^{q-1}-1}\sum_{k=1}^{q-1}t^{k+1}x^{k+1}\cr
&+\left[t^{(q-1)^2}x^{q-1}\!+\!\sum_{\substack{\alpha,\beta\ge 0; \beta\le q-2\cr 0<\alpha+\beta\le q-1}}\left[2^{\alpha+\beta}\binom{2(q-1)\!-\!\alpha\!-\!\beta}{q-1}\!-\!\binom{\alpha+\beta}\alpha\right]t^{(q-1)^2-(\alpha+\beta q)}x^{2(q-1)-(\alpha+\beta)}\right]\cr
&\cdot \sum_{\substack{k\ge 1; j\ge 0\cr k+j\le q-1}}\binom{q-1-k}j(-1)^j(xt)^{q-1-k-j}t^{2k}x^{3k} \qquad \pmod{x^q-x}\cr
\equiv\;&(t-t^{q^2})x+\frac{t^{q^2-1}-1}{t^{q-1}-1}\sum_{k=1}^{q-1}t^{k+1}x^{k+1}\cr
&+\sum_{\substack{k\ge 1; j\ge 0\cr k+j\le q-1}}\binom{q-1-k}j(-1)^jt^{q(q-1)+k-j}x^{q-1+2k-j}\cr
&+\sum_{\substack{\alpha,\beta\ge 0; \beta\le q-2\cr 0<\alpha+\beta\le q-1}}\sum_{\substack{k\ge 1; j\ge 0\cr k+j\le q-1}}\left[2^{\alpha+\beta}\binom{2(q-1)-\alpha-\beta}{q-1}-\binom{\alpha+\beta}\alpha\right]\cr
&\cdot \binom{q-1-k}j(-1)^jt^{q(q-1)+k-j-(\alpha+\beta q)}x^{2(q-1)+2k-j-(\alpha+\beta)}\qquad \pmod{x^q-x}.
\end{split}
\]
Let $x$ vary over $\Bbb F_q$ and sum both sides of the above equation. We have 
\[
\begin{split}
&\sum_{n=1}^{q^2-1}\Bigl(\sum_{a\in\Bbb F_q}a^nd_n(a)\Bigr)t^n\cr
=\;&-\frac{t^{q^2-1}-1}{t^{q-1}-1}t^{q-1}-\sum_{\substack{k\ge1;j\ge 0\cr k+j\le q-1\cr 2k-j\equiv 0\kern -2mm\pmod{q-1}}}\binom{q-1-k}j(-1)^j t^{q(q-1)+k-j}\cr
&-\sum_{\substack{\alpha,\beta\ge 0; \beta\le q-2\cr 0<\alpha+\beta\le q-1}}\sum_{\substack{k\ge1;j\ge 0\cr k+j\le q-1\cr 2k-j\equiv \alpha+\beta \kern -2mm\pmod{q-1}}}
\left[2^{\alpha+\beta}\binom{2(q-1)-\alpha-\beta}{q-1}-\binom{\alpha+\beta}\alpha\right]\cr
&\cdot\binom{q-1-k}j(-1)^jt^{q(q-1)+k-j-(\alpha+\beta q)}.
\end{split}
\]
Thus, for $1\le n\le q^2-1$,
\[
\sum_{a\in\Bbb F_q}a^nd_n(a)=-\delta_n-\text{I}-\text{II}+\text{III},
\]
where $\delta_n$ is defined in \eqref{3.2a} and
\begin{equation}\label{4.2}
\text{I}=\sum_{\substack{k\ge 1;j\ge 0\cr k+j\le q-1\cr 2k-j\equiv 0\kern-2mm\pmod{q-1}\cr q(q-1)+k-j=n}}\binom{q-1-k}j(-1)^j,
\end{equation}
\begin{equation}\label{4.3}
\text{II}=\sum_{\substack{\alpha,\beta\ge 0;\beta\le q-2\cr 0<\alpha+\beta\le q-1}}\sum_{\substack{k\ge 1;j\ge 0\cr k+j\le q-1\cr 2k-j\equiv \alpha+\beta\kern-2mm\pmod{q-1}\cr q(q-1)+k-j-(\alpha+\beta q)=n}} 2^{\alpha+\beta}\binom{2(q-1)-\alpha-\beta}{q-1}\binom{q-1-k}j(-1)^j,
\end{equation}
\begin{equation}\label{4.4}
\text{III}=\sum_{\substack{\alpha,\beta\ge 0;\beta\le q-2\cr 0<\alpha+\beta\le q-1}}\sum_{\substack{k\ge 1;j\ge 0\cr k+j\le q-1\cr 2k-j\equiv \alpha+\beta\kern-2mm\pmod{q-1}\cr q(q-1)+k-j-(\alpha+\beta q)=n}} \binom{\alpha+\beta}\alpha\binom{q-1-k}j(-1)^j.
\end{equation}

We now determine the sums $\text{I}$, $\text{II}$ and $\text{III}$ separately. Write $n=u+vq$, $0\le u,v\le q-1$.

\begin{lem}\label{L4.1}
Assume $1\le n\le q^2-2$. We have
\begin{equation}\label{4.5}
\text{\rm I}=\sum_{\max\{-1,v-q+\frac{u+v}{q-1}\}<s\le v-q+\frac{u+v}{q-1}+1}\binom{u+v-(s-v+q-1)(q-1)}{(s-2v+2q)(q-1)-2(u+v)}.
\end{equation}
\end{lem}

\begin{proof} 
In the sum in \eqref{4.2}, the conditions on $k$ and $j$ can be {\em replaced} by 
\begin{equation}\label{4.5a}
\begin{cases}
0<k\le q-1,\cr
2k-j=s(q-1),\quad 0\le s\le 2,\cr
q(q-1)+k-j=u+vq.
\end{cases}
\end{equation}
(Note. In \eqref{4.5a} we dropped the restriction on $j$. However, this relaxation has no effect on the sum in \eqref{4.2} since only additional zero terms are brought in.) Solving \eqref{4.5a} for $k$ and $j$ in terms of $s$, we get
\[
\begin{cases}
k=(s-v+q)(q-1)-u-v,\cr
j=(s-2v+2q)(q-1)-2(u+v),\vspace{1mm}\cr
\displaystyle \max\{-1,v-q+\frac{u+v}{q-1}\}<s\le v-q+\frac{u+v}{q-1}+1.
\end{cases}
\]
Thus
\[
\text{I}=\sum_{\max\{-1,v-q+\frac{u+v}{q-1}\}<s\le v-q+\frac{u+v}{q-1}+1}\binom{u+v-(s-v+q-1)(q-1)}{(s-2v+2q)(q-1)-2(u+v)}.
\]
\end{proof}
 
\noindent{\bf Note.} The sum in \eqref{4.5} has at most one term. More precisely, when $0<u+v<q-1$,
\begin{equation}\label{4.7a}
\text{I}=0.
\end{equation}
When $q-1\le u+v<2(q-1)$,
\begin{equation}\label{4.7b}
\text{I}=
\begin{cases}
\displaystyle \binom{u+v-(q-1)}{(q+2-v)(q-1)-2(u+v)}&\text{if}\ q-2\le v\le q-1, \vspace{2mm}\cr
0&\text{if}\ v<q-2.
\end{cases}
\end{equation}

\begin{lem}\label{L4.2}
We have
\begin{equation}\label{4.6}
\text{\rm II}=\sum_{\substack{0\le s\le 2\cr \max\{0,v-s+\frac{u+v}{q-1}\}<\alpha\le \min\{q-1,v-s+\frac{u+v}{q-1}+1\}}}\binom{u+v-(s+\alpha-v-1)(q-1)}{(s+2\alpha-2v)(q-1)-2(u+v)}.
\end{equation}
\end{lem}

\begin{proof}
First note that if $0<\alpha+\beta<q-1$, then $\binom{2(q-1)-\alpha-\beta}{q-1}=0$ since the sum $(q-1)+(q-1-\alpha-\beta)$ has carry in base $p$. Thus
\begin{equation}\label{4.7}
\begin{split}
\text{II}\;&=\sum_{\substack{\alpha,\beta\ge 0;\beta\le q-2\cr \alpha+\beta=q-1}}\sum_{\substack{k\ge 1;j\ge 0\cr k+j\le q-1\cr 2k-j\equiv \alpha+\beta\kern-2mm\pmod{q-1}\cr q(q-1)+k-j-(\alpha+\beta q)=n}}
(-1)^j\binom{q-1-k}j\cr
&=\sum_{1\le \alpha\le q-1}\sum_{\substack{k\ge 1;j\ge 0\cr k+j\le q-1\cr 2k-j\equiv 0\kern-2mm\pmod{q-1}\cr k-j+\alpha(q-1)=n}}
(-1)^j\binom{q-1-k}j.
\end{split}
\end{equation}
In \eqref{4.7}, the conditions on $\alpha,k,j$ can be {\em replaced} by
\begin{equation}\label{4.8}
\begin{cases}
0<\alpha\le q-1,\cr
0<k\le q-1,\cr
2k-j=s(q-1),\qquad 0\le s\le 2,\cr
k-j+\alpha(q-1)=u+vq.
\end{cases}
\end{equation} 
Solving \eqref{4.8} for $k$ and $j$ in terms of $s$ and $\alpha$, we have
\[
\begin{cases}
k=(s+\alpha-v)(q-1)-u-v,\cr
j=(s+2\alpha-2v)(q-1)-2(u+v),\cr
0\le s\le 2, \vspace{1mm}\cr
\displaystyle\max\{0, v-s+\frac{u+v}{q-1}\}<\alpha\le\min\{q-1,v-s+\frac{u+v}{q-1}+1\}.
\end{cases}
\]
Thus 
\[
\text{II}=\sum_{\substack{0\le s\le 2\cr \max\{0,v-s+\frac{u+v}{q-1}\}<\alpha\le \min\{q-1,v-s+\frac{u+v}{q-1}+1\}}}\binom{u+v-(s+\alpha-v-1)(q-1)}{(s+2\alpha-2v)(q-1)-2(u+v)}.
\]
\end{proof}

\noindent{\bf Note.}
In \eqref{4.6}, for each $0\le s\le 2$, there is at most one $\alpha$ in the specified range. Hence the sum contains at most three terms. More precisely,
when $0<u+v<q-1$,
\begin{equation}\label{4.9a}
\begin{split}
\text{II}=\;& \sum_{0\le s\le \min\{2,v\}}\binom{u+v}{(-s+2)(q-1)-2(u+v)}\cr
=\;& \sum_{0\le s\le \min\{1,v\}}\binom{u+v}{(-s+2)(q-1)-2(u+v)}\cr
=\;&
\begin{cases}
\displaystyle \binom u{2(q-1)-2u}&\text{if}\ v=0,\vspace{2mm}\cr
\displaystyle \binom{u+v}{2(q-1)-2(u+v)}+\binom{u+v}{q-1-2(u+v)}&\text{if}\ 1\le v\le q-2.
\end{cases}
\end{split}
\end{equation}
When $u+v=q-1$,
\begin{equation}\label{4.9b}
\begin{split}
\text{II}\;&=\sum_{\max\{0,v+3-q\}\le s\le\min\{2,v+1\}}\binom 0{(-s+2)(q-1)}\cr
&=\begin{cases}
0&\text{if}\ v=0,\cr
1&\text{if}\ 1\le v\le q-1.
\end{cases}
\end{split}
\end{equation}
When $q-1< u+v<2(q-1)$,
\begin{equation}\label{4.9c}
\begin{split}
\text{II}\;&=\sum_{\max\{0,v+3-q\}\le s\le\min\{2,v+1\}}\binom{u+v-(q-1)}{(-s+4)(q-1)-2(u+v)}\cr
&=\sum_{\max\{0,v+3-q\}\le s\le 1}\binom{u+v-(q-1)}{(-s+4)(q-1)-2(u+v)}\cr
&=\begin{cases}
\displaystyle \binom{u+v-(q-1)}{4(q-1)-2(u+v)}+\binom{u+v-(q-1)}{3(q-1)-2(u+v)}&\text{if}\ 1\le v\le q-3, \vspace{2mm}\cr
\displaystyle \binom{u-1}{q+1-2u}&\text{if}\ v=q-2, \vspace{2mm}\cr
0&\text{if}\ v=q-1.
\end{cases}
\end{split}
\end{equation}

\begin{lem}\label{L4.3}
We have
\begin{equation}\label{4.9}
\begin{split}
\text{\rm III}=\;&\sum_{\substack{-1\le s\le 1\cr \max\{s-\frac{u+v}{q-1},v-q+1\}\le \epsilon<\min\{s-\frac{u+v}{q-1}+1,v\}}}(-1)^{v+\epsilon}\binom{u+2v-(s-\epsilon)(q-1)-\epsilon}{(s-2\epsilon+1)(q-1)-2u-v-\epsilon}\cr
&-\sum_{\max\{-2,v-q+\frac{u+v}{q-1}\}<s\le\min\{1,v-q+\frac{u+v}{q-1}+1\}}\binom{u+v-(s-v+q-1)(q-1)}{(s-2v+2q)(q-1)-2(u+v)}.
\end{split}
\end{equation}
\end{lem}

\begin{proof}
In \eqref{4.4}, the conditions on $\alpha,\beta,k,j$ can be {\em replaced} by
\begin{equation}\label{4.10}
\begin{cases}
0\le \beta\le q-2,\cr
0<\alpha+\beta\le q-1,\cr
0<k\le q-1,\cr
2k-j-(\alpha+\beta)=s(q-1),\qquad -1\le s\le 1,\cr
k-j-\alpha-u=\epsilon q,\qquad \epsilon\in\Bbb Z,\cr
q-1-\beta-v+\epsilon=0.
\end{cases}
\end{equation}
Solving \eqref{4.10} for $k,j,\beta$ in terms of $\alpha,s,\epsilon$, we have
\[
\begin{cases}
\beta=q-1-v+\epsilon,\cr
k=(s-\epsilon+1)(q-1)-u-v,\cr
j=(s-2\epsilon+1)(q-1)-2u-v-\epsilon-\alpha,\cr
-1\le s\le 1, \vspace{1mm}\cr
\displaystyle \max\{s-\frac{u+v}{q-1},v-q+1\}\le \epsilon<\min\{s-\frac{u+v}{q-1}+1,v\}, \vspace{1mm}\cr
v-\epsilon-q+1<\alpha\le v-\epsilon.
\end{cases}
\]
Therefore,
\begin{equation}\label{4.11}
\begin{split}
\text{III}=\;&\sum_{\substack{-1\le s\le 1\cr \max\{s-\frac{u+v}{q-1},v-q+1\}\le \epsilon<\min\{s-\frac{u+v}{q-1}+1,v\}}}\cr
&\sum_{v-\epsilon-q+1<\alpha\le v-\epsilon}(-1)^{v+\epsilon+\alpha}\binom{\alpha+q-1-v+\epsilon}\alpha \binom{u+v-(s-\epsilon)(q-1)}{(s-2\epsilon+1)(q-1)-2u-v-\epsilon-\alpha}\cr
=\;&\sum_{\substack{-1\le s\le 1\cr \max\{s-\frac{u+v}{q-1},v-q+1\}\le \epsilon<\min\{s-\frac{u+v}{q-1}+1,v\}}}\cr
&\sum_{v-\epsilon-q+1<\alpha\le v-\epsilon}(-1)^{v+\epsilon}\binom{v-\epsilon}\alpha \binom{u+v-(s-\epsilon)(q-1)}{(s-2\epsilon+1)(q-1)-2u-v-\epsilon-\alpha}\cr
&\kern 9cm \text{(by Lemma~\ref{L2.1})}\cr
=\;&\sum_{\substack{-1\le s\le 1\cr \max\{s-\frac{u+v}{q-1},v-q+1\}\le \epsilon<\min\{s-\frac{u+v}{q-1}+1,v\}}}(-1)^{v+\epsilon}\left[\binom{u+2v-(s-\epsilon)(q-1)-\epsilon}{(s-2\epsilon+1)(q-1)-2u-v-\epsilon}\right.\cr
&-\sum_{0\le \alpha\le v-\epsilon-q+1}\binom{v-\epsilon}\alpha \left. \binom{u+v-(s-\epsilon)(q-1)}{(s-2\epsilon+1)(q-1)-2u-v-\epsilon-\alpha}\right]\cr
=\;&\sum_{\substack{-1\le s\le 1\cr \max\{s-\frac{u+v}{q-1},v-q+1\}\le \epsilon<\min\{s-\frac{u+v}{q-1}+1,v\}}}(-1)^{v+\epsilon}\binom{u+2v-(s-\epsilon)(q-1)-\epsilon}{(s-2\epsilon+1)(q-1)-2u-v-\epsilon}-A,
\end{split}
\end{equation}
where
\begin{equation}\label{4.11a}
\begin{split}
A=\;&\sum_{\substack{-1\le s\le 1\cr \max\{s-\frac{u+v}{q-1},v-q+1\}\le \epsilon<\min\{s-\frac{u+v}{q-1}+1,v\}}}(-1)^{v+\epsilon}\cr
&\cdot\sum_{0\le \alpha\le v-\epsilon-q+1}\binom{v-\epsilon}\alpha \binom{u+v-(s-\epsilon)(q-1)}{(s-2\epsilon+1)(q-1)-2u-v-\epsilon-\alpha}.
\end{split}
\end{equation}
Note that in \eqref{4.11a}, $v-q+1\le \epsilon$ in the outer sum; thus the inner sum is empty unless $\epsilon=v-q+1$. Therefore
\begin{equation}\label{4.12}
A=\sum_{\max\{-2,v-q+\frac{u+v}{q-1}\}<s\le\min\{1,v-q+\frac{u+v}{q-1}+1\}}\binom{u+v-(s-v+q-1)(q-1)}{(s-2v+2q)(q-1)-2(u+v)}.
\end{equation}
Equation~\eqref{4.9} follows from \eqref{4.11} and \eqref{4.12}. 
\end{proof}

\noindent{\bf Note.}
In \eqref{4.9}, the first sum has at most three terms and the second sum has at most one term. The formula for III can be made more explicit as we saw earlier in similar situations. We assume that $q>3$.
When $0<u+v<q-1$,
\begin{equation}\label{4.20}
\begin{split}
\text{III}\;&=\sum_{-1\le s\le\min\{0,v-1\}}(-1)^{v+s}\binom{u+2v-s}{(-s+1)(q-1)-2u-v-s}\cr
&=\begin{cases}
\displaystyle -\binom{u+1}{2(q-1)-2u+1}&\text{if}\ v=0,\vspace{2mm}\cr
\displaystyle(-1)^{v+1}\binom{u+2v+1}{2(q-1)-2u-v+1}+(-1)^v\binom{u+2v}{q-1-2u-v}&\text{if}\ 1\le v\le q-2.
\end{cases}
\end{split}
\end{equation}
When $q-1\le u+v<2(q-1)$,
\begin{equation}\label{4.21}
\begin{split}
\text{III}=\;&\sum_{\max\{-1,v-q+2\}\le s\le\min\{1,v\}}(-1)^{v+s+1}\binom{u+2v-q-s+2}{(-s+3)(q-1)-2u-v-s+1}\cr
&-\sum_{\max\{-1,v-q+2\}\le s\le v-q+2}\binom{u+v-(s-v+q-1)(q-1)}{(s-2v+2q)(q-1)-2(u+v)}\cr
=\;&\begin{cases}
0&\text{if}\ v=0, \vspace{3mm}\cr
\displaystyle (-1)^v\binom{u+2v-q+3}{4(q-1)\!-\!2u\!-\!v\!+\!2}+(-1)^{v+1}\binom{u+2v-q+2}{3(q-1)\!-\!2u\!-\!v\!+\!1}\vspace{2mm}\cr
\displaystyle +(-1)^v\binom{u+2v-q+1}{2(q-1)-2u-v}&\text{if}\ 1\le v\le q-3, \vspace{3mm}\cr
\displaystyle \binom{u+q-2}{2q-2u}-\binom{u+q-3}{q-2u}-\binom{u-1}{2q-2u}&\text{if}\ v=q-2, \vspace{3mm}\cr
\displaystyle \binom{u+q-1}{q-2u-1}-\binom u{q-2u-1}&\text{if}\ v=q-1.
\end{cases}
\end{split}
\end{equation}

\medskip
 
To recap, we have the following proposition.

\begin{prop}\label{P4.4}
Let $q$ be an odd prime power and let $1\le n\le q^2-2$. Then
\begin{equation}\label{4.13}
\sum_{a\in\Bbb F_q}a^nd_n(a)=-\delta_n-\text{\rm I}-\text{\rm II}+\text{\rm III},
\end{equation}
where $\delta_n$ defined in \eqref{3.2a}; {\rm I, II, III} are given by \eqref{4.5}, \eqref{4.6} and \eqref{4.9} respectively.
\end{prop}

\begin{thm}\label{T4.5}
Let $q$ be an odd prime power and let $1\le n\le q^2-2$. Write $n=u+vq$, $0\le u,v\le q-1$, and write $3n\equiv u'+v'q>0\pmod{q^2-1}$, $0\le u'v'\le q-1$. Then
\begin{equation}\label{4.14}
\begin{split}
\sum_{a\in\Bbb F_q}d_n(a)^3=\;&3(-\delta_n-\text{\rm I}-\text{\rm II}+\text{\rm III})\cr
&+\begin{cases}
\displaystyle -2^{-u'-v'}\binom{u'+v'}{q-1}+\binom{2(q-1)\!-\!u'\!-\!v'}{q-1-u'}&\text{if}\ 1\le u'\!+\!v'q\le q^2\!-\!2, \vspace{2mm}\cr
-1&\text{if}\ u'+v'q=q^2-1.
\end{cases}
\end{split}
\end{equation}
\end{thm}

\begin{proof}
This follows from \eqref{2.2}, \eqref{4.13} and \eqref{2.8}.
\end{proof}

\begin{cor}\label{C4.6}
In Theorem~\ref{T4.5} assume $q>3$ and $(q,n)$ is desirable. Then $u+v\ne q-1$ and in $\Bbb F_p$,
\begin{equation}\label{4.15}
3(\text{\rm I}+\text{\rm II}-\text{\rm III})=
\begin{cases}
\displaystyle -2^{-u'-v'}\binom{u'+v'}{q-1}+\binom{2(q-1)\!-\!u'\!-\!v'}{q-1-u'}&\text{if}\ 1\le u'+v'q\le q^2\!-\!2,\vspace{2mm}\cr
-1&\text{if}\ u'+v'q=q^2-1.
\end{cases}
\end{equation}
\end{cor}

\begin{proof}
By Theorems 2.3 and 2.4 of \cite{HL}, $(n,q-1)\le 3<q-1$. Thus $u+v\ne q-1$. Now \eqref{4.15} follows from \eqref{4.14}.
\end{proof}


\end{document}